\documentclass[12pt]{article}
\usepackage{amsmath,amssymb,theorem,color} %
\newcommand{\donothing}[1]{{}}
\numberwithin{equation}{section}
\newtheorem{theorem}{Theorem}[section]
\newtheorem{corollary}[theorem]{Corollary}
\newtheorem{lemma}[theorem]{Lemma}

\newtheorem{conjecture}[theorem]{Conjecture}
\newenvironment{proof}{{\bf Proof}.\ }{ \hfill $\square$}
\newcommand{\bke}[1]{\left ( #1 \right )}

\newcommand{\norm}[1]{\left \| #1 \right \|}
\newcommand{\bka}[1]{{\langle #1 \rangle}}

\def\be{\beta}

\def\de{\delta}

\def\ve{\varepsilon}
\def\e {\varepsilon}

\def\th{\theta}

\def\la{\lambda}
\def\si{\sigma}

\def\ph{\varphi} %

\def\De{\Delta}

\def\Om{\Omega}

\newcommand{\R}{\mathbb{R}}

\renewcommand{\div}{\mathop{\rm div}}

\newcommand{\supp} {\mathop{\mathrm{supp}}}

\newcommand{\pd}{\partial}
\newcommand{\nb}{\nabla}
\newcommand{\td}{\tilde}

\newcommand{\lec}{{\ \lesssim \ }}

\newcommand{\bs}{\backslash\, }
\newcommand{\mz}{\backslash \{0\}}

\begin{document}

\title{Point singularities of 3D stationary Navier-Stokes flows}

\author{Hideyuki Miura%
\thanks{Department of Mathematics, Osaka University,
Osaka, Japan. Email: miura@math.sci.osaka-u.ac.jp}, \quad Tai-Peng Tsai%
\thanks{Department of Mathematics, University of British Columbia,
Vancouver, BC V6T 1Z2, Canada. Email: ttsai@math.ubc.ca}}

\date{May 25, 2009}

\maketitle

\begin{abstract}

This article characterizes the singularities of very weak solutions of
 3D stationary Navier-Stokes equations in a punctured ball which are
sufficiently small in weak $L^3$.

{\it Keywords}: stationary Navier-Stokes equations, point singularity,
very weak solution, Landau solution.
\end{abstract}

\section{Introduction}

We consider point singularities of very weak solutions of the 3D
stationary Navier-Stokes equations in a finite region $\Om$ in $\R^3$. The
Navier-Stokes equations for the velocity $u : \Om\to \R^3$ and pressure
$p : \Om \to \R$ with external force $f : \Om \to \R^3$ are
\begin{equation}
\label{NS}
 -\De u + (u \cdot \nb) u  + \nb p =f, \quad \div u = 0, \quad (x \in
\Om).
\end{equation}
A {\it very weak solution} is a vector function $u$ in
$L^2_{loc}(\Om)$ which satisfies \eqref{NS} in distribution sense:
\begin{equation}
  \int- u \cdot \De \ph + u_j u_i \pd_j \ph_i = \bka{f,\ph}, \quad
  \forall
\ph \in C^\infty_{c,\si} (\Om),
\end{equation}
and $\int u \cdot \nb h = 0$ for any $h \in C^\infty_c(\Om)$.
Here the force $f$ is allowed to be a distribution and
\begin{equation}
C^\infty_{c,\si} (\Om) = \{ \ph \in C^\infty_{c} (\Om,\R^3) : \div \ph=0\}.
\end{equation}
In this definition the pressure is not needed.
Denote $B_R = \{ x \in \R^3: |x|<R\}$ and $B_R^c = \R^3 \bs
B_R$ for $R>0$.

We are concerned with the behavior of very weak solutions which
solve \eqref{NS} in the punctured ball $B_2 \mz$ with zero force,
i.e., $f=0$.
There are a lot of studies on this problem \cite{Dyer-Edmunds,
Sh1,Sh2,Choe-Kim,Kim-Kozono}. A typical result is to show that, under
some conditions, the solution is a very weak solution across the
origin without singular forcing supported at the origin (removable
singularity), and is regular, i.e., locally bounded, under possibly
more assumptions (regularity).  Dyer-Edmunds \cite{Dyer-Edmunds}
proved removable singularity and regularity assuming both $u,p \in
L^{3+\ve} (B_2)$ for some $\ve>0$. Shapiro \cite{Sh1,Sh2} proved
removable singularity and regularity assuming $u \in L^{3+\ve} (B_2)$
for some $\ve>0$ and $u(x) = o(|x|^{-1})$ as $x \to 0$, without
assumption on $p$. Choe and Kim \cite{Choe-Kim} proved removable
singularity assuming $u \in L^{3} (B_2)$ {\it or} $u(x) = o(|x|^{-1})$
as $x \to 0$, and regularity assuming $u \in L^{3+\ve} (B_2)$ for some
$\ve>0$. 
Kim and Kozono \cite{Kim-Kozono} recently proved removable
singularity under the same assumptions as \cite{Choe-Kim}, and
regularity assuming $u \in L^3(B_2)$ or $u$ is small in 
weak $L^3$. As mentioned in \cite{Kim-Kozono}, their
result is optimal in the sense that if their assumption is replaced
by
\begin{equation}
\label{u-bound}
|u(x)| \le C_* |x|^{-1}
\end{equation}
for $0<|x|<2$, then the singularity is not removable in general, due
to the existence of $\textit{Landau solutions}$, which is the family
of explicit singular solutions calculated by L. D. Landau  in 1944
\cite{Landau}, and can be found in standard textbooks, see e.g.,
\cite[p. 82]{LL} or \cite[p. 206]{Batchelor}.

The purpose of this article is to characterize the singularity
and to identify the leading order behavior of very weak solutions
satisfying the threshold assumption (\ref{u-bound})
when the constant $C_*$ is sufficiently small.
We show that it is given by Landau solutions.

We now recall Landau solutions in order to state our main theorems.
Landau solutions can be parametrized by vectors $b \in \R^3$ in the
following way: For each $b\in \R^3$ there exists a unique
($-1$)-homogeneous solution $U^b$ of \eqref{NS} together with an
associated pressure $P^b$ which is ($-2$)-homogeneous, such that
$U^b, P^b$ are smooth in $\R^3 \mz$ and they solve
\begin{equation}
\label{u-eq} -\De u + (u \cdot \nb) u + \nb p =b\de, \quad \div u =
0,
\end{equation}
in $\R^3$ in the sense of distributions, where
$\de$ denotes the Dirac $\delta$ function.
When $b=(0,0,\be)$ with $\be\ge 0$,
they have the following explicit formulas in spherical coordinates
$r,\th,\phi$ with $x=(r\sin \th \cos \phi, \ r \sin \th \sin \phi, \
r \cos \th)$:
\begin{equation}
 U = \frac 2r\bke{\frac {A^2-1}{(A-\cos \th)^2} -1} e_r - \frac
 {2\sin \th} {r(A-\cos \th)} e_\th, \quad
P =\frac {-4(A\cos \th -1)}{r^2(A-\cos \th)^2}
\end{equation}
where $e_r = \frac x{r}$ and $e_\th = (\cos \th \cos \phi, \ \cos \th
\sin \phi, \ -\sin \th)$. The parameters $\be \ge 0$ and $A \in
(1,\infty]$ are related by the formula
\begin{equation}
\be = 16 \pi \bke{A + \frac 12 A^2 \log \frac {A-1}{A+1} + \frac
  {4A}{3(A^2-1)}}.
\end{equation}
The formulas for general $b$ can be obtained from rotation.
One checks directly that $\norm{rU^b}_{L^\infty}$ is monotone in $|b|$ and
$\norm{rU^b}_{L^\infty} \to 0$ (or $\infty$) as $|b| \to 0$ (or
$\infty$).
Recently Sverak \cite{Sverak} proved that Landau solutions are
the only solutions of \eqref{NS} in $\R^3 \mz$
which are smooth and ($-1$)-homogeneous in $\R^3 \mz$, without
assuming axisymmetry.
See also \cite{TX,Cannone,Korolev-Sverak} for related results.

If $u, p$ is a solution of (1.1), we will denote by
\begin{equation}
T_{ij}(u,p)= p\de_{ij} + u_i u_j - \pd_i u_j - \pd_j u_i
\end{equation}
the momentum flux density tensor in the fluid,
which plays an important role to determine the equation
for $(u,p)$ at 0.
Our main result is the following.

\begin{theorem}
\label{theorem1} For any $q\in (1,3)$, there is a small
$C_*=C_*(q)>0$ such that, if $u$ is a very weak solution of
\eqref{NS} with zero force in $B_2 \mz$ satisfying \eqref{u-bound}
in $B_2 \mz$, then there is a scalar function $p$ satisfying
$|p(x)|\le C|x|^{-2}$, unique up to a constant, so that $(u,p)$
satisfies \eqref{u-eq} in $B_2$ with $b_i = \int_{|x|=1} T_{ij}(u,p)
n_j(x)$, and
\begin{equation}
\label{thm1-1}
\norm{u - U^b}_{ W^{1,q}(B_1)} + \sup _{x \in B_1} |x|^{3/q-1} |(u -
U^b)(x)| \le
C C_*,
\end{equation}
where the constant $C$ is independent of $q$ and $u$.
\end{theorem}

The exponent $q$ can be regarded as the degree of
the approximation of $u$ by $U^b$.
The closer $q$ gets to $3$, the less singular $u-U^b$ is.
But in our theorem,
$C_*(q)$ shrinks to zero as $q \to 3_-$.  Ideally, one would like to
prove that $u - U^b \in L^\infty$. However, it seems quite subtle in
view of the following model equation for a scalar function,
\begin{equation}
-\De v + c v =0, \quad c = \De v /v.
\end{equation}
If we choose $v = \log |x|$, then $c(x) \in L^{3/2}$ and $\lim_{|x|\to
0} |x|^2 |c(x)|=0$, but $v\not \in L^\infty$. In equation \eqref{NSw}
for the difference $w=u-U^b$, there is a term $(w \cdot \nb) U^b$ which
has similar behavior as $cv$ above.

In fact, we have the following stronger result. Denote by $L^r_{wk}$
the weak $L^r$ spaces. We claim the same conclusion as in Theorem
\ref{theorem1} assuming only a small $L^3_{wk}$ bound of $u$ but not
the pointwise bound \eqref{u-bound}.
\begin{theorem}
\label{theorem2} There is a small $\e_*>0$ such that, if $u$ is a
very weak solution of \eqref{NS} with zero force in $\Om=B_{2.1}
\mz$ satisfying $\norm{u}_{L^3_{wk}(\Om)}=:\e \le \e_*$, then $u$
satisfies $|u(x)| \le C_1 \e |x|^{-1}$ in $B_{2}\mz$ for some $C_1$.
Thus the conclusion of Theorem \ref{theorem1} holds if $C_1 \e \le
C_*(q)$.
\end{theorem}



%
%
%


%
%

Our results are closely related to the \textit{regularity problem} of
very weak solutions, which  could be considered when $u$ is only
assumed to be in $L^2_{loc}$. In fact, the problem with the
assumption $u$ being large in $L^3_{wk}$ already exhibits a great
difficulty. Recall the scaling property of \eqref{NS}: If $(u,p)$ is
a solution of \eqref{NS}, then so is
\begin{equation}
\label{rescaled}
(u_\la,p_\la)(x) = (\la u(\la x), \la^2 p(\la x)), \quad (\la >0).
\end{equation}
The known methods are primarily perturbation arguments.  Since
$L^3_{wk}$-quasi-norm is invariant under the above scaling and does not
become smaller when restricted to smaller regions, one would need to
exploit the structure of the Navier-Stokes equations in order to get a
positive answer. Compare the recent result \cite{CSTY} on axisymmetric
solutions of nonstationary Navier-Stokes equations, which also
considers a borderline case under the natural scaling.

This work is inspired by Korolev-Sverak \cite{Korolev-Sverak} in which
they study the asymptotic as $|x| \to \infty$ of solutions of
\eqref{NS} satisfying \eqref{u-bound} in $\R^3 \bs B_1$. They show
that the leading behavior is also given by Landau solutions if $C_*$
is sufficiently small.  Our theorem can be considered as a dual
version of their result.  However, their proof is based on the unique
existence in $\R^3$ of the equation for $v=\ph (u-U^b)+\zeta$ where
$\ph$ is a cut-off function supported near infinity and $\zeta$ is a
suitable function chosen to make $\div v=0$. If one tries the same
approach for our problem, since one needs to remove the origin as well
as the region $|x|\ge 2$ while extending $u-U^b$, one needs to choose
a sequence $\ph_k$ with the supports of $1-\ph_k$ shrinking to the
origin, which produce very singular force terms near the origin.

Instead, we first prove Lemma \ref{lemma3} which gives the equation
for $(u,p)$ near the origin. Since the equation for $u$ is same as
$U^b$ near the origin for $b=b(u)$, the $\delta$-functions at the
origin cancel in the equation for their difference. We then apply the
approach of Kim-Kozono \cite{Kim-Kozono} to the difference equation,
and prove its unique existence  in $W^{1,r}_0 (B_2)$ for $3/2\le
r<3$ and uniqueness in $W^{1,r}_0\cap L^3_{wk}(B_2)$ for $1< r<3/2$,
which improves the regularity of the original difference.
Above $W^{1,r}_0(B_2)$ is the closure of $C^{\infty}_c (B_2)$ in the
$W^{1,r}(B_2)$-norm.

As an application, we give the following corollary. Recall $u_\la$
for $\la >0$ is defined in \eqref{rescaled}. A solution $u$ on $B_2
\mz$ is called {\it discretely self-similar} if there is a $\la_1
\in (0,1)$ so that $u_{\la_1}=u$. Such a solution is completely
determined by its values in the annulus $B_1 \bs B_{\la_1}$ since
$u(\la_1 ^k x) = \la_1^{-k}u(x)$. They contain minus-one homogeneous
solutions as a special subclass.

\begin{corollary}
If $u$ satisfies the assumptions of Theorem \ref{theorem1} and
furthermore $u$ is discretely self-similar in $B_2\mz$, then
$u\equiv U^b$.
\end{corollary}

This corollary also follows from \cite{Korolev-Sverak} (with domain
$\R^3 \bs B_1$ and $\la_1 >1$).  In the case of small $C_*$, this
corollary extends the result of Sverak \cite{Sverak} on minus-one
homogeneous solutions. The classification of discretely self-similar
solutions with large $C_*$ is unknown.

As another application, we consider a conjecture by Sverak
\cite[\S 5]{Sverak}:
\begin{conjecture}
\label{conj}
If $u$ is a solution of the stationary Navier-Stokes
equations \eqref{NS} with zero force in $\R^3 \mz$ satisfying
\eqref{u-bound} with some $C_* >0$. Then $u$ is a Landau solution.
\end{conjecture}

We give a partial answer for this problem.
\begin{corollary}
\label{th1-3}
Conjecture \ref{conj} is true, provided the constant $C_*$ is
sufficiently small.
\end{corollary}

The above corollary can be also shown to be true by either our main theorem
or the result of Korolev-Sverak \cite{Korolev-Sverak}, see section
\ref{sec3.3}. The corresponding conjecture  for large $C_*$
is related to the regularity problem of evolutionary
Navier-Stokes equations via the usual blow-up procedures.

\section{Preliminaries}

In this section we collect some lemmas for the proof of Theorem
\ref{theorem1}.  The first lemma recalls H\"{o}lder 
and Sobolev type inequalities in Lorentz spaces.
We denote the Lorentz spaces by $L^{p,q}$
($1<p<\infty$, $1\le q \le \infty$).
Note $L^3_{wk}=L^{3,\infty}$.

\begin{lemma}
\label{lemma1}
Let $B=B_2\subset \R^n$, $n \ge 2$.

i) \quad Let $1<p_1, p_2< \infty$ with $1/p:=1/p_1 +1/p_2 <1$ and
let $1\le r_1, r_2\le \infty$. For $f \in L^{p_1,r_1}$ and $g \in
L^{p_2,r_2}$, we have
\begin{equation}
\|fg\|_{L^{p,r}(B)} \le C\|f\|_{L^{p_1,r_1}(B)}
\|g\|_{L^{p_2,r_2}(B)} \qquad \textit{for} \ r:=\textit{min}\{r_1,
r_2 \},
\end{equation}
where $C=C(p_1,r_1,p_2,r_2)$.
\\
ii) \quad Let $1<r<n$. For $f\in W^{1,r}(B)$, we have
\begin{equation}
\|f\|_{L^{\frac{nr}{n-r},r}(B)} \le C\| f \|_{W^{1,r}(B)},
\end{equation}
where $C=C(n,r)$.

\end{lemma}

Part (i) of Lemma \ref{lemma1} was proved in \cite{ONeil}. Part (ii)
was proved in \cite{ONeil} for $\R^n$ and in
\cite{Kozono-Yamazaki,Kim-Kozono} for bounded domains.

By this lemma, when $n=3$ and $1<r<3$, we have
\begin{equation} \label{lemma1-1}
  \norm{fg}_{L^r(B)} \le C \norm{f}_{L^3_{wk}} \norm{g}_{L^{\frac
  {3r}{3-r},r}} \le C_r \norm{f}_{L^3_{wk}(B)} \norm{g}_{W^{1,r}(B)}.
\end{equation}
This estimate first appeared in \cite{Kim-Kozono} and plays an important role
for our application.

The next lemma is on interior estimates for Stokes system with no
assumption on the pressure.

\begin{lemma}
\label{lemma2}
Assume $v\in L^1$ is a distribution solution of the Stokes
system
\begin{equation}
-\De v_i + \pd_i p = \pd_j f_{ij}, \quad \div v=0 \quad \text{in
 }B_{2R}
\end{equation}
and $f \in L^r$ for some $r \in (1,\infty)$.  Then $v\in
W^{1,r}_{loc}$ and, for some constant $C_r$ independent of $v$ and
$R$,
\begin{equation}
\norm{\nb v}_{L^r(B_R)} \le C_r  \norm{f}_{L^r(B_{2R})} + C_r R^{-4+3/r}
\norm{v}_{L^1(B_{2R})}.
\end{equation}
\end{lemma}

This lemma is \cite{Sverak-Tsai}, Theorem 2.2. Although the statement
in \cite{Sverak-Tsai} assumes $v\in W^{1,r}_{loc}$, its proof
only requires $v \in L^1$.  This lemma can be also considered as
\cite[Lemma A.2]{CSTY} restricted to time-independent functions.

The following lemma shows the first part of Theorem \ref{theorem1},
except \eqref{thm1-1}. In particular, it shows that $(u,p)$ solves
\eqref{u-eq}.

\begin{lemma}
\label{lemma3} If $u$ is a very weak solution of \eqref{NS} with
zero force in $B_2 \mz$ satisfying \eqref{u-bound} in $B_2 \mz$
(with $C_*$ allowed to be large), there is a scalar function $p$
satisfying $|p(x)|\le C|x|^{-2}$, unique up to a constant, such that
$(u,p)$ satisfies \eqref{u-eq} in $B_2$ with $b_i = \int_{|x|=1}
T_{ij}(u,p) n_j(x)$. Moreover, $u,p$ are smooth in $B_2 \mz$.
\end{lemma}

\begin{proof}
For each $R \in (0,1/2]$, $u$ is a very weak solution in $B_2 - \bar
B_R$ in $L^\infty$. Lemma \ref{lemma2} shows $u$ is a weak solution in
$W^{1,2}_{loc}$. The usual theory shows that $u$ is smooth and there
is a scalar function $p_R$, unique up to a constant, so that $(u,p_R)$
solves \eqref{NS} in $B_2 - \bar B_R$, see e.g.~\cite{Galdi-2}.  By
the scaling argument in Sverak-Tsai \cite{Sverak-Tsai} using Lemma
\ref{lemma2}, we have for $x \in B_{3R}- B_{2R}$,
\begin{align}
\label{u+p}
|\nabla^k u (x)| \le  \frac{C_kC_*}{|x|^{k+1}}
\qquad \textrm{for} \ k=1,2,\ldots,
\end{align}
where $C_k=C_k(C_*)$ are independent of $R \in (0,1/2]$ and its
dependence on $C_*$ can be dropped if $C_* \in (0,1)$. Varying $R$,
\eqref{u+p} is valid for $x \in B_{3/2} \mz$. For $0<R<R'$, by
uniqueness of $p_R'$, the difference $p_R|_{B_2-\bar B_{R'}}-p_{R'}$
is a constant.  Thus we can fix the constant by requiring $p_R =
p_{1/2}$ in $B_2 \bs \bar B_{1/2}$, and define $p(x) = p_R(x)$ for any
$x \in B_2 \mz$ with $R =|x|/2$. By the equation, $|\nb p(x)| \le C
C_* |x|^{-3}$.  Integrating from $|x|=1$ we get $|p(x)| \le CC_*
|x|^{-2}$.  In particular
\begin{equation}
\label{T}
|T_{ij}(u,p)(x)| \le CC_* |x|^{-2}
\qquad \textrm{for} \ x \in B_{3/2}\mz.
\end{equation}

Denote $NS(u) = - \De u + (u \cdot \nb)u + \nb p$. We have $NS(u)_i =
\pd_j T_{ij}(u)$ in the sense of distributions. Thus,
by divergence theorem and $NS(u)=0$ in $B_2 \mz$,
\begin{equation}
\label{b-int}
b_i = \int_{|x|=1} T_{ij}(u,p) n_j(x) = \int_{|x|=R} T_{ij}(u,p) n_j(x)
\end{equation}
for any $R \in (0,2)$.
Let $\phi$ be any test function in $C_c^{\infty}(B_1)$.
For small $\e>0$,
\begin{align*}
\langle NS(u)_i, \phi \rangle
&=
-\int T_{ij}(u) \pd_j \phi
\\
&=
-\int_{B_1\backslash B_{\varepsilon}} T_{ij}(u) \pd_j \phi
-
\int_{B_{\varepsilon}} T_{ij}(u) \pd_j \phi
\\
&=\int_{B_1\backslash B_{\varepsilon}} \pd_j T_{ij}(u) \phi
+
\int_{\partial B_{\varepsilon}}T_{ij}(u) \phi n_j
-\int_{\partial B_1} T_{ij}(u) \phi n_j
-\int_{B_{\ve}}T_{ij}(u)\pd_j \phi.
\end{align*}
In the last line, the first integral is zero since $NS(u)=0$ and the
third integral is zero since $\phi=0$.  By the pointwise estimate
(\ref{T}), the last integral is bounded by $C\ve^{3-2}$.  On the other
hand, by \eqref{b-int},
\begin{equation}
\int_{\partial B_{\varepsilon}}T_{ij}(u) \phi n_j
\rightarrow
b_i \phi(0) \quad \mathrm{as} \quad \ve \rightarrow 0.
\end{equation}
Thus $(u,p)$ solves \eqref{u-eq} and we have proved the lemma.
\end{proof}

\medskip

It follows from the proof that $|b| \le C C_*$ for $C_* <1$.
With this lemma, we have completely proved Theorem \ref{theorem1} in
the case $q<3/2$.  In the case $3/2 \le q <3$, it remains to prove
\eqref{thm1-1}.

\section{Proof of main theorem}
In this section, we present the proof of Theorem \ref{theorem1}. We
first prove that solutions belong to $W^{1,q}$. We next apply this
result to obtain the pointwise estimate.  For what follows, denote
\begin{equation}
 w=u-U, \quad U = U^b,
\end{equation}
where $U^b$ is the Landau solution with $b$ given by \eqref{b-int}.
By Lemma \ref{lemma3}, there is a function $\td p$ such that $(w,\td p)$
satisfies in $B_2$ that
\begin{equation}\label{NSw}
\begin{split}
-\De w + U \cdot \nb w + w \cdot \nb (U + w) + \nb \td p = 0 ,&\quad
\div w =0  ,
\\
|w(x)| \le \frac {CC_*} {|x|}, \quad |\td p(x)| \le \frac {CC_*} {|x|^2}.&
\end{split}
\end{equation}
Note that the $\de$-functions at the origin cancel.

\subsection{$W^{1,q}$ regularity}
In this subsection we will show $w\in W^{1,q}(B_1)$.  Fix a cut off
function $\ph$ with $\ph =1$ in $B_{9/8}$ and $\ph = 0$ in
$B_{11/8}^c$.  We localize $w$ by introducing
\begin{equation}
v = \ph w +\zeta
\end{equation}
where $\zeta$ is a solution of the problem $\div \zeta=- \nb \ph \cdot
w$.  By Galdi \cite[Ch.3]{Galdi-1}  Theorem 3.1, there exists such a
$\zeta$ satisfying
\begin{equation}\label{zeta-est}
\supp \zeta \subset B_{3/2} \backslash B_1, \quad \|\nabla \zeta
\|_{L^{100}} \le C\|\nabla \ph \cdot w\|_{L^{100}} \le CC_*.
\end{equation}
The vector $v$ is supported in $\bar B_{3/2}$, satisfies $v \in
W^{1,r}\cap L^3_{wk}$ for $r<3/2$ by \eqref{u-bound}, \eqref{u+p} and
\eqref{zeta-est}, and
\begin{equation}
-\De v + U \cdot \nb v + v \cdot \nb (U + v) + \nb \pi = f, \quad \div
v=0,
\label{NSv}
\end{equation}
where $\pi = \ph \td p$, and
\begin{equation}
\begin{split}
f= -2(\nb \ph \cdot \nb) w - (\Delta \ph) w+ (U\cdot \nb \ph) w
+(\ph^2-\ph)w \cdot \nb w 
+ (w\cdot \nb \ph) w 
\\ + \td p \nb \ph
- \De \zeta + (U \cdot \nb) \zeta + \zeta \cdot \nb (U + \ph w +
\zeta) + \ph w \cdot \nb \zeta
\end{split}
\end{equation}
is supported in the annulus $\bar B_{3/2} \bs B_{1}$.  One verifies
directly that, for some $C_1$,
\begin{equation}
\label{f-est}
\sup _{1 \le r \le 100} \norm{f}_{W_0^{-1,r}(B_2)}  \le C_1 C_*.
\end{equation}

Our proof is based on the following lemmas.
\begin{lemma}[Unique existence]
\label{lem3.1}
For any $3/2 \le r<3$, for sufficiently small $C_*=C_*(r)>0$, there is a
unique solution $v$ of \eqref{NSv} and \eqref{f-est} in the set
\begin{equation}
V=\{ v \in W^{1,r}_0(B_2), \quad \norm{v}_V:=
\norm{v}_{W^{1,r}_0(B_2)} \le C_2 C_* \}
\end{equation}
for some $C_2>0$ independent of $r \in [3/2,3)$.
\end{lemma}

\begin{lemma}[Uniqueness]
\label{lem3.2} Let $1<r< 3/2$.
If both $v_1$ and $v_2$ are solutions of \eqref{NSv} and \eqref{f-est} in
$W^{1,r}_0\cap L^3_{wk}$ and $C_*+\norm{v_1}_{L^3_{wk}}
+\norm{v_2}_{L^3_{wk}}$ is sufficiently small, then $v_1 = v_2$.
\end{lemma}

Assuming the above lemmas, we get $W^{1,q}$ regularity as
follows. First we have a solution $\tilde v$ of \eqref{NSv} in
$W^{1,q}_0(B_2)$ by Lemma \ref{lem3.1}. On the other hand, both $v=\ph
w +\zeta$ and $\tilde v$ are small solutions of \eqref{NSv} in
$W^{1,r}_0\cap L^3_{wk}(B_2)$ for $r=5/4$, and thus $v=\tilde v$ by
Lemma \ref{lem3.2}. Thus $v \in W^{1,q}_0(B_2)$ and $w \in
W^{1,q}(B_1)$.

\medskip
{\bf Proof of Lemma \ref{lem3.1}.}\quad
Consider the following mapping $\Phi$: For each $v\in V$, let $\bar v
= \Phi v$ be the unique solution in $W^{1,r}_0(B_2)$ of the Stokes
system
\begin{equation}
-\De \overline{v} + \nb \bar \pi
 = f - \nb \cdot (U \otimes v
+v \otimes (U+v)),
\quad
\div \overline{v}  =0.
\end{equation}
By estimates for the Stokes system, see Galdi \cite[Ch.4]{Galdi-1}
Theorem 6.1, in particular (6.9), for $1<r<\infty$, we have
\begin{equation}
\|\bar v\|_{{W}^{1,r}_0(B_2)}
\le
N_r \|f\|_{W_0^{-1,r}}
+N_r \|\nb \cdot (U \otimes v
+ v \otimes (U+v))\|_{W_0^{-1,r}}
\end{equation}
for some constant $N_r>0$ which is uniformly bounded for $r$ in any
compact regions of $(1,\infty)$.
By \eqref{f-est} and Lemma \ref{lemma1}, in particular
\eqref{lemma1-1}, for $1<r<3$,
\begin{equation}
\begin{split}
\|\bar v\|_{{W}^{1,r}_0(B_2)} 
&\le
N_r C_1 C_*
+N_r \|U \otimes v +v \otimes (U+v) \|_{L^r}
\\
&\le N_r C_1 C_* + N_r C_r(\|U
\|_{L^3_{wk}}+\|v \|_{L^3_{wk}}) \|v\|_V.
\end{split}
\end{equation}
We now choose $C_2 = 2 (C_1+1)\sup_{3/2\le r< 3}N_r$.
Since $V \subset L^3_{wk}$ if $r \ge 3/2$, we get $\bar v = \Phi v \in
V$ if $C_*$ is sufficiently small.

We next consider the difference estimate. Let $v_1,v_2 \in V$, $\bar
v_1 = \Phi v_1$, and $\bar v_2 = \Phi v_2$.  Then
\begin{equation}
\label{diff-est}
\| \Phi v_1 - \Phi v_2 \|_{W^{1,r}}
\le
CC_r(\|U \|_{L^3_{wk}}
+\|v_1 \|_{L^3_{wk}}+\|v_2 \|_{L^3_{wk}} )\|v_1-v_2\|_{W^{1,r}}.
\end{equation}
Taking $C_*$ sufficiently small for $3/2\le r<3$, we get $\norm{\Phi
v_1 - \Phi v_2}_V \le \frac 12 \norm{v_1 - v_2}_V$, which shows that
$\Phi$ is a contraction mapping in $V$ and thus has a unique fixed
point.  We have proved the unique existence of the solution for
\eqref{NSv}--\eqref{f-est} in $V$.  \hfill $\square$

\medskip

{\it Remark.}\quad Since the constant $C_r$ from Lemma \ref{lemma1}
(ii) blows up as $r \to 3_-$, our $C_*$ shrinks to zero as $r \to
3_-$.

\medskip
{\bf Proof of Lemma \ref{lem3.2}.}\quad
By the difference estimate \eqref{diff-est}, we have
\begin{equation}
\| v_1 - v_2 \|_{W^{1,r}} \le C(\|U \|_{L^3_{wk}} +\|v_1
\|_{L^3_{wk}}+\|v_2 \|_{L^3_{wk}} )\|v_1-v_2\|_{W^{1,r}}.
\end{equation}
Thus, if $C(\|U \|_{L^3_{wk}} +\|v_1
\|_{L^3_{wk}}+\|v_2 \|_{L^3_{wk}})<1$, we conclude $v_1 = v_2$.
\hfill $\square$

\subsection{Pointwise bound}
In this subsection, we will prove pointwise bound of $w$ using
$\norm{w}_{W^{1,q}}\lec C_*$.

For any fixed $x_0 \in B_{1/2}\mz$, let $R=|x_0|/4$ and $E_k =
 B(x_0,kR)$, $k=1,2$.

Note $q^* \in (3,\infty)$.  Let $s$ be the dual exponent of $q^*$,
$1/s+1/q^* = 1$. We have
\begin{equation}
\norm{w}_{L^1(E_2)} \lec \norm{w}_{L^{q^*}(E_2)} \norm{1}_{L^{s}(E_2)}
\lec C_* R^{4-3/q}.
\end{equation}
By the interior estimate Lemma \ref{lemma2},
\begin{equation}
\norm{\nb w}_{L^{q^*}(E_1)} \lec \norm{f}_{L^{q^*}(E_2)} +
R^{-4+3/q^*}\norm{w}_{L^1(E_2)}
\end{equation}
where $f=U \otimes w + w \otimes (U+w)$. Since $|U|+|w| \lec C_*
|x|^{-1} \lec C_* R^{-1}$ in $E_2$,
\begin{equation}
\norm{f}_{L^{q^*}(E_2)} \lec C_* R^{-1} \norm{w}_{L^{q^*}(E_2)} \lec
C_*^2 R^{-1}.
\end{equation}
We also have $R^{-4+3/q^*}\norm{w}_{L^1(E_2)} \lec R^{-4+3/q^*}
C_* R^{4-3/q} = C_* R^{-1}$.
Thus
\begin{equation}
\norm{\nb w}_{L^{q^*}(E_1)} \lec C_* R^{-1}.
\end{equation}
By Gagliardo-Nirenberg inequality in $E_1$,
\begin{equation}
\norm{w}_{L^{\infty}(E_1)} \lec \norm{ w}_{L^{q^*}(E_1)}^{1-\th}
\norm{\nb w}_{L^{q^*}(E_1)}^\th + R^{-3} \norm{w}_{L^1(E_1)},
\end{equation}
where $1/\infty=(1-\th)/q^* + \th (1/q_* -1/3)$ and thus $\th= 3/q-1$.
We conclude $\norm{w}_{L^{\infty}(E_1)} \le C_* R^{-\th}$. Since $x_0$
is arbitrary, we have proved the pointwise bound, and completed the
proof of Theorem  \ref{theorem1}.

\medskip

{\it Remark.} Equivalently, one can define $v(x) = u(x_0+Rx)$, find
the equation of $v$, estimate $v$ in $L^\infty(B_1)$, and then derive
the bound for $w(x_0)$.

\subsection{Proof of Theorem \ref{theorem2}}

In this subsection we prove Theorem \ref{theorem2}. For any $x_0 \in
B_{2}\mz$, let $v(x)= \la u(\la x+x_0)$ with $\la = \min(0.1 ,
|x_0|)/2$.  By our choice of $\la$, $v$ is a very weak solution in
$B_2$ and $\norm{v}_{L^3_{wk}(B_2)} \le \e =
\norm{u}_{L^3_{wk}(B_{2.1}\mz)} $.  By \cite{Kim-Kozono}, we have
$\norm{v}_{L^\infty(B_1)} \le C_2 \e$ for some constant $C_2$ if $\e$
is sufficiently small. Thus $|u(x_0)| \le C_2\e \la^{-1}\le 40 C_2 \e
|x_0|^{-1}$.

\subsection{Proof of Corollary \ref{th1-3}}
\label{sec3.3}
In this subsection we prove Corollary \ref{th1-3}.  Suppose $u$
satisfies \eqref{u-bound} with $C_*=C_*(q)$, $q=2$, given in Theorem
\ref{theorem1}.  Let $b$ be given by \eqref{b-int},
$U=U^b$ and $w=u-U$. Let $u_\la= \la u(\la x)$ be the rescaled
solution and $w_\la(x) = \la w(\la x)$. Note $U$ is scaling-invariant.
Then $u_\la = U + w_\la$ also satisfies \eqref{u-bound} with same
$C_*$. By Theorem \ref{theorem1} with $q=2$, we have the bound
\begin{equation}
|w_\la(x)| \le C C_* |x|^{-1/2}, \quad |x| <1,
\end{equation}
which is uniform in $\la$. In terms of $w$ and $y=\la x$, we get
\begin{equation}
|w(y)| \le C C_* \la^{-1} |\la^{-1} y|^{-1/2}, \quad |y| \le \la.
\end{equation}
Now fix $y$ and let $\la \to \infty$. We conclude $w \equiv 0$.

\section*{Acknowledgments}
We thank S.~Gustafson, K.~Kang, and K.~Nakanishi for fruitful
discussions.  We thank V.~Sverak for suggesting Corollary
\ref{th1-3}. We thank H.~Kim for many helpful comments. We also thank
the referees for many useful comments.
 Miura acknowledges the hospitality of the University of
British Columbia.  The research of Miura was partly supported by the
JSPS grant no.~191437.  The research of Tsai is partly supported by
Natural Sciences and Engineering Research Council of Canada, grant
no.~261356-08.

\end{document}